\documentclass[11pt,reqno]{amsart}
\usepackage[top=1in, bottom=1in, left=1in, right=1in]{geometry}

\usepackage[utf8]{inputenc}
\usepackage{amsfonts,amsthm,amssymb,amsmath}
\usepackage[english]{babel}
\usepackage[colorlinks=true]{hyperref}  
\newcommand{\E}{\mathbb E}%expectation
\newcommand{\R}{\mathbb R}%reals
\newcommand{\W}{\mathrm W}%Wass distance
\newcommand{\TR}{\mathrm{Tr}}%Trace
\newcommand{\IM}{\mathcal I}%Fisher information matrix
\newcommand{\D}{\delta}%log-Sob deficit
\newcommand{\C}{\mathrm{cov}}%covariance
\newcommand{\ID}{\mathrm{Id}_n}%n by n identity matrix
\newtheorem{theorem}{Theorem}
\newtheorem{lemma}[theorem]{Lemma}
\newtheorem{question}[theorem]{Question}

\newtheorem{proposition}[theorem]{Proposition}
\newtheorem{corollary}[theorem]{Corollary}
\theoremstyle{definition}

\theoremstyle{remark}
\newtheorem{remark}{Remark}

\begin{document}

\title{stability of the logarithmic Sobolev inequality via the F\"ollmer Process}
\author{Ronen Eldan}
\thanks{Ronen Eldan is the incumbent of the Elaine Blond Career development chair, supported by a European Research Council Staring Grant (ERC StG) and by an Israel Science Foundation grant no. 715/16.}
\address{Department of Mathematics,
Weizmann Institute of Science,
Rehovot 76100, Israel} 
\email{ronen.eldan@weizmann.ac.il}

\author{Joseph Lehec}
\address{Ceremade (UMR CNRS 7534),
Université Paris-Dauphine,
75016 Paris and 
DMA (UMR CNRS 8553),
École Normale Supérieure,
75005 Paris, France}
\email{lehec@ceremade.dauphine.fr}

\author{Yair Shenfeld}
\thanks{Yair Shenfeld was supported in part by NSF grants CAREER-DMS-1148711 and  DMS-1811735, ARO through PECASE award W911NF-14-1-0094, and the Simons 
Collaboration on Algorithms \& Geometry. His residence at MSRI was supported by NSF grant DMS-1440140.}
\address{Sherrerd Hall 323, Princeton University, Princeton, NJ
08544, USA}
\email{yairs@princeton.edu}

\maketitle

\begin{abstract}
We study the stability and instability of the Gaussian logarithmic Sobolev inequality, in terms of covariance, Wasserstein distance and Fisher information, addressing several open questions in the literature. We first establish an improved logarithmic Sobolev inequality which is at the same time scale invariant and dimension free. As a corollary, we show that if the covariance of the measure is bounded by the identity, one may obtain a sharp and dimension-free stability bound in terms of the Fisher information matrix. We then investigate under what conditions stability estimates control the covariance, and when such control is impossible. For the class of measures whose covariance matrix is dominated by the identity, we obtain optimal dimension-free stability bounds which show that the deficit in the logarithmic Sobolev inequality is minimized by Gaussian measures, under a fixed covariance constraint. On the other hand, we construct examples showing that without the boundedness of the covariance, the inequality is not stable. Finally, we study stability in terms of the Wasserstein distance, and show that even for the class of measures with a bounded covariance matrix, it is hopeless to obtain a dimension-free stability result. The counterexamples provided motivate us to put forth a new notion of stability, in terms of proximity to mixtures of the Gaussian distribution. We prove new estimates (some dimension-free) based on this notion. These estimates are strictly stronger than some of the existing stability results in terms of the Wasserstein metric. Our proof techniques rely heavily on stochastic methods.
\\
\emph{keywords:} Quantitative functional inequalities, stochastic methods. 
\end{abstract}

\section{Introduction}

\subsection{Overview} The logarithmic Sobolev inequality is one of the fundamental Gaussian functional inequalities \cite{ledoux}. The inequality was proven independently in the information-theoretic community by Stam~\cite{stam} and in the mathematical-physics community by Gross \cite{gross}. The form of the inequality which we consider in this paper states that for any nice enough probability measure $\mu$ on $\R^n$,
\begin{equation}\label{eq:lsi}
H(\mu\mid \gamma)\le \frac{1}{2}I(\mu\mid \gamma).
\end{equation}
Here $\gamma$ is the standard Gaussian measure on $\mathbb R^n$ with density
\[
\gamma(dx)=(2\pi)^{-\frac{n}{2}}e^{-\frac{\vert x\vert^2}{2}}dx,
\]
and $H(\mu\mid \gamma),~I(\mu\mid \gamma)$ are the relative entropy and relative Fisher information respectively: 
\[
H(\mu\mid\gamma)=\int_{\R^n}\log \left(\frac{d\mu}{d\gamma}\right)d\mu ,
\]
and 
\[
I(\mu\mid\gamma)=\int_{\R^n}\left\vert\nabla\log \left(\frac{d\mu}{d\gamma}\right)\right\vert^2d\mu.
\]
The inequality~\eqref{eq:lsi} is sharp as can be seen by taking $\mu$ to be any translation of $\gamma$, and in fact these are the only equality cases as was proved in \cite{carlen}. This characterization naturally leads to the question of stability. That is, supposing that the deficit
\[
\D ( \mu )  := \frac 12 I(\mu\mid \gamma ) - H ( \mu \mid \gamma) 
\] 
is small, in what sense is $\mu$ close to a translate of $\gamma$? The study of stability questions for Gaussian inequalities is an ongoing active area of research with many applications \cite{FIL}, \cite{MN}. The precise notion of stability is context-dependent, but a common thread is the desire to make the stability estimates dimension-free. This is because the Gaussian measure itself is inherently infinite-dimensional, so we expect functional inequalities about Gaussian measures in $\R^n$ to extend to infinite dimensions. Indeed, the infinite-dimensional nature of the logarithmic Sobolev inequality is crucial to its applications to quantum field theory, which was the original motivation of Gross. For example, it was proven in a series of works \cite{CFMP}, \cite{MN, MN1}, \cite{eldan}, \cite{BBJ} that the Gaussian isoperimetric inequality (which implies the log-Sobolev inequality) enjoys such dimension-free estimates. The logarithmic Sobolev inequality however, turns out to be much more delicate.

\subsection{Fisher information matrix and deficit} Our first observation is that the log-Sobolev inequality can be self-improved in a dimension-free way. This observation then leads to natural stability results, provided that $\C(\mu)\preceq \ID$. Let us formulate first the log-Sobolev inequality in an alternative way.  Define the entropy and Fisher information of $\mu$ with respect to the Lebesgue measure by
\[
H ( \mu \mid \mathcal L ) = \int_{\R^n} \log \left( \frac {d \mu} {dx} \right) \, d\mu , 
\]
and 
\[
I( \mu \mid \mathcal L ) = \int_{\R^n} \left\vert\nabla\log \left( \frac {d \mu} {dx} \right)\right\vert^2 \, d\mu.
\]
The log-Sobolev inequality \eqref{eq:lsi} then reads 
\[
H ( \mu \mid \mathcal L ) - H( \gamma \mid \mathcal L )\leq\frac{1}{2}\left(I( \mu \mid \mathcal L )-n\right).
\]
It is well known (see for instance the very end of~\cite{carlen}) that the above inequality can be improved via scaling. Let $X\sim\mu$ and let $\sigma>0$. Computing the entropy and Fisher information of the law of $\sigma X$, and optimizing over $\sigma$, shows that 
\begin{equation}\label{eq:lsiDim2}
H(\mu\mid \mathcal L) - H( \gamma \mid \mathcal L ) \leq \frac{n}{2} 
\log \left( \frac { I( \mu\mid\mathcal L ) } n \right).
\end{equation} 
Inequality \eqref{eq:lsiDim2} is known as the \emph{dimensional logarithmic Sobolev inequality}. Our first result shows that this bound is sub-optimal, and that one should consider the individual eigenvalues of the Fisher information matrix:
\[
\IM ( \mu \mid \mathcal L ) := \int_{\R^n}
\left(\nabla \log \left(\frac{d\mu}{dx}\right)\right)^{\otimes 2} \, d\mu.
\]
This matrix is of course related to the Fisher information via $\TR[\IM ( \mu \mid \mathcal L )]=I( \mu \mid \mathcal L)$.
\begin{theorem}\label{thm:fisherEig}
Let $\mu$ be a probability measure on $\R^n$. Then  
\begin{equation}
\label{eq:fisherEig1}
H(\mu\mid\mathcal L) - H ( \gamma \mid \mathcal L ) \leq \frac12 \log \det \left[\mathcal I ( \mu \mid \mathcal L )\right]. 
\end{equation}
\end{theorem}
Theorem \ref{thm:fisherEig} improves upon~\eqref{eq:lsiDim2} by the AM/GM inequality. 
Note also that~\eqref{eq:fisherEig1} is at the same time scale invariant 
and dimension-free: both sides of the inequality behave additively when taking tensor products.
\begin{remark} After the first version of this work was released, we realized that Theorem~\ref{thm:fisherEig} had already been obtained by Dembo, see~\cite{dembo}, at the beginning of page 12. Its application to the stability of the logarithmic Sobolev inequality, see Corollary~\ref{cor:fisherEig} and Theorem~\ref{thm:cov} below, appears to be new. 
\end{remark}
\begin{remark} A reverse form of Theorem~\ref{thm:fisherEig} is known when the measure is log-concave. Observe first the integration by parts identity
\[
\mathcal I ( \mu \mid \mathcal L ) = - \int_{\R^n}
\nabla^2 \log \left(\frac{d\mu}{dx}\right) \, d\mu.
\]
The reverse form of Theorem~\ref{thm:fisherEig} then asserts that if $\mu$ is log-concave and if $\log \det$ is moved inside the integral in the right-hand side of~\eqref{eq:fisherEig1}, then the inequality is reversed, see~\cite{AKSW}. See also~\cite{CFGLSW} 
for a simpler proof based on the functional Santal\'o inequality.  
\end{remark}
\bigskip

Self-improvements of the form of \eqref{eq:lsiDim2} and \eqref{eq:fisherEig1} lead to stability results for the log-Sobolev inequality, provided that the covariance of $\mu$ is bounded by the identity. Define the function $\Delta(t):=t-\log(1+t)$ for $t>-1$. It was observed in \cite{BGRS} that if $\E_\mu [\vert x\vert^2]\le n$, then \eqref{eq:lsiDim2} implies that
\begin{equation}\label{eq:lsiDimCov}
\D(\mu)\ge \frac{n}{2}\Delta\left(\frac{I(\mu\mid\gamma)}{n}\right).
\end{equation}
From \eqref{eq:lsiDimCov} one can deduce weaker but more amenable stability statements. For example,
\begin{equation}\label{eq:lsiDimWas}
\D(\mu)\ge \frac{c}{n}\W_2^4(\mu,\gamma) 
\end{equation}
for some universal constant $c$, see~\cite{BGRS} for the details. Here, $\W_2(\mu,\gamma)$ is the Wasserstein two-distance between $\mu$ and $\gamma$. In general, the $p$-Wasserstein distance ($p\ge 1$) for probability measures $\mu,\nu$ is defined as 
\begin{equation}\label{eq:wasserstein}
\W_p(\mu,\nu)
:=\inf_{X,Y}  \left\{ \E [\, \vert X-Y \vert^p ]^{1/p} \right\},
\end{equation}
where the infimum is taken over all couplings $(X,Y)$ of $(\mu,\nu)$. A problematic feature of both bounds, \eqref{eq:lsiDimCov} and \eqref{eq:lsiDimWas}, is that they are dimension-dependent: Letting formally $n$ tend to $+\infty$, we see that the lower bound on the deficit tends to $0$ in both cases (observe that $\Delta (\epsilon)\sim \epsilon^2/2$ when $\epsilon$ tends to $0$). Note also that the log-Sobolev deficit behaves additively when taking tensor products, and that neither of the two lower bounds~\eqref{eq:lsiDimCov} and~\eqref{eq:lsiDimWas} does. In particular if $\mu$ is the product of a $1$-dimensional measure by a $(n-1)$-dimensional standard Gaussian, the lower bound is of order $1/n$ in both cases, whereas the deficit is of order $1$. On the other hand, we can deduce from Theorem \ref{thm:fisherEig} the following dimension-free estimate. 
\begin{corollary}\label{cor:fisherEig}
Let $\mu$ be a probability measure on $\R^n$ such that $\E_\mu [x^{\otimes 2}]\preceq \ID$, and let $\{ \beta_i \}_{i=1}^n$ be the 
eigenvalues of its \textbf{Gaussian} Fisher information matrix $\mathcal I ( \mu \mid \gamma )$. 
Then 
\begin{equation}\label{eq:fisherEig2}
\D ( \mu ) \geq \frac12\sum_{i=1}^n \Delta(\beta_i).
\end{equation}
\end{corollary}
Again, by concavity of the logarithm, \eqref{eq:fisherEig2} is a strict improvement on \eqref{eq:lsiDimCov}.

To see how Corollary~\ref{cor:fisherEig} follows from Theorem~\ref{thm:fisherEig}, 
note that~\eqref{eq:fisherEig1} can be rewritten as
\begin{equation}\label{eq:fisherEig3}
\D ( \mu ) \geq \frac 12 \sum_{i=1}^n \Delta ( \alpha_i - 1 )  ,   
\end{equation}
where $\alpha_1,\dotsc,\alpha_n$ are the eigenvalues of the Fisher 
information matrix of $\mu$ with respect to the Lebesgue measure. 
Using the integration by parts identity
\[
\mathcal I ( \mu \mid \mathcal L ) - \ID = \mathcal I ( \mu \mid \gamma ) + 
\ID - \E_\mu [ x^{\otimes 2} ]  
\]
we see that if $\E_\mu [ x^{\otimes 2} ] \preceq \ID$, then 
\[
\mathcal I ( \mu \mid \mathcal L ) - \ID \succeq \mathcal I ( \mu \mid \gamma ) \succeq 0 . 
\] 
Since $\Delta$ is increasing on $[0,+\infty)$, the inequality~\eqref{eq:fisherEig2} thus follows 
from~\eqref{eq:fisherEig3}. 
\begin{remark}
Corollary \ref{cor:fisherEig} bears an interesting formal resemblance to the following result. Let $T$ be the Brenier map from $\mu$ to $\gamma$ and let $\{\kappa_i(x)\}_{i=1}^n$ be the eigenvalues of the map $DT(x)-\ID$. Then it can be shown \cite{cordero} that 
\begin{equation*}
\D ( \mu ) \geq\sum_{i=1}^n\E_{\mu}[\Delta(\kappa_i)].
\end{equation*}
For further appearances of the map $\Delta$ as a cost function in transportation distance, see \cite{BGRS}.
\end{remark}  

Let us note that although Theorem \ref{thm:fisherEig} (and thus Corollary \ref{cor:fisherEig}) follow from a simple scaling argument (see section \ref{sec:self}), it is arguably the only natural dimension-free stability result that has minimal assumptions on $\mu$. To the best of our knowledge, the only other known dimension-free estimates of the form of Corollary \ref{cor:fisherEig} are the results of \cite{FIL}, which impose strong conditions of the measure $\mu$, namely that it satisfies a Poincar\'e inequality. 

\subsection{Covariance and Gaussian mixtures} As we saw, in order to get stability estimates for the deficit from the self-improvements of the log-Sobolev inequality, we need to assume that $\E_\mu [\vert x\vert^2]\le n$. The phenomenon that the size of $\C(\mu)$ serves as a watershed for stability estimates has already been observed in the literature, but the precise connection has remained unclear. Indeed, \cite{LNP} raises the question regarding the relation between the distance of the covariance of $\mu$ from the identity, and the possible lower bounds on the deficit. Our next result completely settles this question. 
\begin{theorem}\label{thm:cov}
Let $\mu$ be a probability measure on $\R^n$ and let $\lambda:=\{\lambda_i\}_{i=1}^n$ be the eigenvalues of ${\rm cov}(\mu)$.  Then 
\begin{equation}\label{eq:cov}
\D ( \mu ) \geq \frac 12\sum_{i=1}^n \mathbf{1}_{ \{\lambda_i < 1 \} } ( \lambda_i^{-1} - 1 + \log  \lambda_i ).
\end{equation}
In particular, if $\C(\mu) \preceq \ID$, then
\[
\D(\mu) \geq   \frac 12\sum_{i=1}^n ( \lambda_i^{-1} - 1 + \log  \lambda_i )=\D (\gamma_\lambda)
\]
where $\gamma_\lambda$ is a Gaussian measure on $\R^n$ having the same covariance matrix as $\mu$.

On the other hand, this becomes completely wrong if we remove the hypothesis on the covariance matrix,
even in dimension $1$: there exists a sequence $(\mu_k)$ of mixtures of Gaussian measures on $\R$ such that $\mathrm{var}(\mu_k) \to \infty$ while $\D ( \mu_k )\to 0$.
\end{theorem}
The moral of Theorem \ref{thm:cov} is, that if $\C(\mu) \preceq \ID$, then the deficit $\D ( \mu )$ controls the distance of ${\rm cov}(\mu)$ to the identity. For example, a weaker bound which can be deduced from \eqref{eq:cov} using $\frac{1}{x}-1+\log x\ge \frac{1}{2}(x-1)^2$ for $x\in(0,1]$ is,
\[
\D ( \mu )\ge\frac{1}{4}\Vert {\rm cov}(\mu)- \ID \Vert_{HS}^2,
\]
where the norm on the right hand side is the Hilbert-Schmidt norm. 
On the other hand, if the covariance of $\mu$ is not a priori bounded by the identity, then one can have an arbitrarily small deficit with arbitrarily large variance.

\begin{remark}
Theorem \ref{thm:cov} can also be phrased as a statement about minimizing the deficit subject to a covariance constraint. For simplicity let us consider the one-dimensional situation. Fix a scalar $\sigma>0$. Of all distributions $\mu$ with variance $\sigma$, which one minimizes $\D(\mu)$? Theorem \ref{thm:cov} shows that the answer is dramatically different depending on whether or not $\sigma$ is greater than 1. (If $\sigma=1$ then obviously $\mu=\gamma$ minimizes $\D(\mu)$.) If $\sigma<1$, then the minimizer is the Gaussian measure 
with variance $\sigma$. On the other hand, if $\sigma>1$, then by taking $\mu$ to be an appropriate mixture of Gaussians, we can make $\D(\mu)$ smaller than the Gaussian with variance $\sigma$.
\end{remark}

The Gaussian mixtures in Theorem \ref{thm:cov} served as counterexamples to stability estimates in terms of the distance of $\C(\mu)$ from the identity. In fact, such mixtures show the impossibility of many other stability estimates: 

\begin{theorem}\label{thm:unstable}
For $m\in\R^n$ let $\gamma_{m,\rm Id}$ be the Gaussian measure centered at $m$ with identity covariance matrix. There exists a sequence $(\mu_k)$ of probability measures on $\R$, each of which is a mixture of two 
Gaussian measures of variance $1$, satisfying $\D ( \mu_k ) \to 0$ and 
\[
\lim_{k\to \infty} \inf_{m\in \R} \left\{  \W_1 ( \mu_k , \gamma_{m,\rm Id} ) \right\} = +\infty . 
\]
Additionally, there exists a sequence of dimensions $n(k) \uparrow +\infty$ 
and a sequence $(\mu_k)$ of \textbf{isotropic} (i.e. centered with identity as covariance) measures on $\R^{n(k)}$, satisfying $\D ( \mu_k ) = O ( n(k)^{-1/3} )\to 0$ and 
\[
\inf_{m\in \R^{n(k)}} \left\{ \W_2 ( \mu_k , \gamma_{m,\rm Id} ) \right\} 
= \Omega ( n(k)^{1/6} ) \to  +\infty. 
\]
\end{theorem}
The first statement shows that the log-Sobolev inequality is unstable for $\W_1$, even in dimension $1$. 
The second statement shows that even for isotropic measures, there is no dimension-free 
stability result for $\W_2$. Note however that our second counterexample does not work for $\W_1$; 
as far as we know it could still be the case that $\D(\mu) \geq c \W_1 (\mu,\gamma)^2$ for 
every isotropic $\mu$ on $\R^n$. (Recall that by Jensen's inequality we have $\W_1(\mu,\nu)\leq \W_2 (\mu,\nu)$.)
Explicit counterexamples to stability were discussed recently in the literature, see \cite{kim}. These examples however are complicated and require a lot of tedious computations while ours are completely elementary. We just observe that Gaussian mixtures have small log-Sobolev deficit, see Proposition \ref{prop:shannon} below. Similar Gaussian mixture examples can be found in the context of stability of the entropy power inequality, see \cite{CFP} and references therein.

\begin{proposition}\label{prop:shannon}
Let $p$ be a discrete measure on $\R^n$ and let $S(p) = -\sum p(x) \log p(x)$ be its Shannon entropy. 
Then 
\[
\D( p * \gamma ) \leq S(p) . 
\]
\end{proposition}

\subsection{Decompositions into mixtures}
If we take stock of the results in the  preceding sections, we see that while a result of the form 
\[
\D(\mu)\geq \frac{c}{n}\W_2^4(\mu,\gamma)
\] 
holds under the assumption that $\E_\mu[\vert x\vert^2]\le n$, we cannot replace the right hand side by $\frac{c'}{\sqrt{n}}\W_2^3(\mu,\gamma)$, let alone $c''\W_2^2(\mu,\gamma)$. (These bounds increase in strength since $\W_2^2(\mu,\gamma)\leq 2n$ under the assumption $\E_\mu[\vert x\vert^2]\le n$.) As we saw, mixtures of Gaussians pose counterexamples to such bounds. Our next result shows that in a certain sense, these counterexamples are the only obstacles.
\begin{theorem}\label{thm:dim}
Let $\mu$ be a probability measure on $\R^n$. Then there exists a measure $\nu$ on $\R^n$ such that
\begin{equation}\label{eq:dimension}
\D ( \mu ) \geq \frac1{15}  \frac{\W_2^3 ( \mu , \nu * \gamma ) }{ \sqrt n },
\end{equation}
and so that $\nu$ is a Dirac point mass whenever $\D(\mu)=0$.
\end{theorem}

In fact, that a small deficit implies that $\mu$ is close to being a mixture of Gaussians, is an implication which comes out naturally from our stochastic proof technique as we will see below. The relation between approximate equality in the log-Sobolev inequality and proximity to mixtures of \emph{product} measures, appears in a recent work of Austin \cite{austin} in a more abstract setting of product spaces. Given Theorem \ref{thm:dim} and Proposition~\ref{prop:shannon} we pose the following question.
\begin{question}\label{qu:mainQuestion}
Given a probability measure $\mu$ on $\R^n$, is it true that there exists 
a discrete probability measure $p$ on $\R^n$ satisfying $S(p)\leq C \, \D(\mu)$ and  
\[ 
\W_2^2 ( \mu , p * \gamma ) \leq C \, \D ( \mu ), 
\]
where $C$ is a universal constant?
\end{question} 
Note that both sides of the inequality above behave additively when taking tensor products.
The inequality is thus completely dimension-free, which is our main motivation for it. 

\begin{remark}
While the Wasserstein distance is a bona fide distance between probability measures, in the context of the log-Sobolev inequality it seems more natural to work with lower bounds which are expressed in terms of relative entropy and relative Fisher information. Thus one may wonder, whether it is possible to replace the lower bound on the deficit in Question \ref{qu:mainQuestion} by the relative entropy or Fisher information between $\mu$ and a mixture of Gaussians. We focus on the Wasserstein two-distance distance since by the log-Sobolev and Talagrand's inequalities such results are weaker. Moreover, our decomposition results are easier to prove for the Wasserstein distance. 
\end{remark}

As a step towards answering this question, we prove that an estimate similar in spirit does indeed hold. We show that a random vector distributed like $\mu$, can be written as the sum of two random vectors which are orthogonal in expectation, one of which is close to a Gaussian in a dimension-free way. 
\begin{theorem} \label{thm:uncor}
Let  $\mu$ be a probability measure on $\R^n$ and let $X \sim \mu$. There exists a decomposition $X \stackrel{D}{=} Y + W$ with the property that $\E[ \langle Y, W \rangle ]= 0$, such that
\[
\D(\mu)\geq \frac 12\W_2^2 (\nu, \gamma)
\]
where $Y\sim\nu$. 
\end{theorem}
Theorem \ref{thm:uncor} can be seen as an improvement on \eqref{eq:lsiDimWas}.
Indeed, assume that $\E_\mu [\vert x\vert^2]\le n$.
The theorem implies that 
\[
\W_2 ( \mu ,\gamma ) \leq \W_2 ( \mu , \nu ) + \W_2 ( \nu , \gamma ) 
\leq \E [\vert W\vert^2 ]^{1/2} + \sqrt{ 2 \D(\mu) } . 
\]
Moreover, since $\E[\langle Y,W \rangle] = 0$, we have 
\[
\E [|W|^2] = \E [|X|^2] - \E [|Y|^2] \leq n - \E[|Y|^2] .
\]
If $\D( \mu ) \geq C n$, then~\eqref{eq:lsiDimWas} holds trivially, so 
we can assume additionally that $\D( \mu ) = O (n)$. Then by the theorem $\W_2 ( \nu , \gamma ) = O ( \sqrt n )$
and thus 
\[
\E [ \vert Y\vert^2 ] \geq n - C \sqrt n \, \W_2 ( \nu , \gamma )  \geq n - C \sqrt {2 n \D(\mu) } . 
\]
Putting everything together, we get 
\[
\W_2(\mu, \gamma) \leq \sqrt{ 2 \D(\mu)} + C'\, n^{1/4} \D(\mu)^{1/4} \leq C'' \, n^{1/4} \D(\mu)^{1/4} ,
\]
which is~\eqref{eq:lsiDimWas}.

\subsection{Methods}

We provide two sets of proofs for Theorems \ref{thm:fisherEig} and \ref{thm:cov}. The first set of proofs proceeds by establishing Theorem \ref{thm:fisherEig} via a scaling argument, and then deduces the first part of Theorem \ref{thm:cov} from Theorem~\ref{thm:fisherEig} via the Cram\'er-Rao bound. The second set of proofs uses a stochastic process known as the \emph{Schr\"odinger bridge},  or the \emph{F\"ollmer process}, depending on the context. This process is entropy-minimizing and is thus suitable for the logarithmic Sobolev inequality. For example, it is used in~\cite{lehec} to give a simple proof of the log-Sobolev inequality (see section~\ref{sec:follmer}), and in~\cite{eldan1} to obtain a \emph{reversed} form (see also~\cite{EL}). We use this process to prove Theorems~\ref{thm:dim} and \ref{thm:uncor} as well. Some of our arguments are essentially semigroup proofs (see \cite{LNP}), phrased in a stochastic language, which uses the semigroup of the F\"ollmer process rather than the more common heat or Ornstein-Uhlenbeck semigroups. A key point in our proofs is that we essentially compute two derivatives of the entropy rather than one. This gives us more precise information about the log-Sobolev inequality. The stochastic formulation allows for relatively simple computations. We go however an additional step beyond semigroup techniques, and also analyze \emph{pathwise} behavior of the F\"ollmer process. This analysis provides us with a natural way of decomposing the measure $\mu$ (see the proofs of Theorem \ref{thm:dim} and Theorem \ref{thm:uncor}). 

\subsection{Organization of paper} In section \ref{sec:self} we give the first set of proofs of Theorems \ref{thm:fisherEig} and \ref{thm:cov}. In section \ref{sec:follmer} we define the F\"ollmer process and analyze its properties. This analysis provides us with ways of decomposing $\mu$. Section \ref{sec:comp} contains the second set of  proofs of Theorems \ref{thm:fisherEig}  and \ref{thm:cov} via the  F\"ollmer process, and section \ref{sec:mix} contains the proofs of Theorems \ref{thm:dim} and \ref{thm:uncor}. Finally, the counterexamples to stability (and the proof of Theorem \ref{thm:unstable}) are discussed in section \ref{sec:counter}. 

\subsection{Acknowledgments} We are grateful to Ramon van Handel for his enlightening comments, in particular, the observation that the first part of Theorem \ref{thm:cov} can be deduced from Theorem \ref{thm:fisherEig} via the Cram\'er-Rao bound, is due to him. Our original proof of the first part of Theorem \ref{thm:cov} can be found in section \ref{sec:comp}. We are grateful to Djalil Chafa\"i for bringing reference~\cite{dembo} to our attention. We would also like to thank Max Fathi, Michel Ledoux and Dan Mikulincer for discussions and suggestions. In addition we would like to thank an anonymous referee for useful comments which improved the quality of this manuscript. Finally, we would like to acknowledge the hospitality of MSRI and to thank the organizers of the program on \emph{Geometric Functional Analysis and Applications} in the fall 2017 where part of this work was done. 

\section{self-improvements of the log-Sobolev inequality}\label{sec:self}
In this section we show how Theorem~\ref{thm:fisherEig} and the first part of Theorem~\ref{thm:cov} follow from scaling the log-Sobolev inequality appropriately and the Cram\'er-Rao bound. Recall that the log-Sobolev inequality can be rewritten 
\begin{equation}\label{eq:Elsi}
H(\mu\mid\mathcal L) - H ( \gamma \mid \mathcal L ) \le \frac 12 \left( I( \mu \mid \mathcal L ) - n \right) .
\end{equation}
Let $\Sigma$ be an $n\times n$ symmetric positive definite matrix and let $\mu_{\Sigma}$ be the law of $\Sigma X$ where $X\sim \mu$. Easy computations show that  
\[
H(\mu_{\Sigma}\mid\mathcal L)=H(\mu\mid\mathcal L)- \log \det \Sigma
\quad \text{and} \quad \mathcal I ( \mu_{\Sigma} \mid \mathcal L ) = \Sigma^{-1} \mathcal I ( \mu \mid \mathcal L ) \Sigma^{-1} . 
\]
In particular 
\[
I ( \mu_{\Sigma} \mid \mathcal L )  = \TR \left( \Sigma^{-2}  \mathcal I ( \mu \mid \mathcal L ) \right) . 
\]
Applying~\eqref{eq:Elsi} to $\mu_\Sigma$ thus yields 
\[
H ( \mu \mid \mathcal L ) - H ( \gamma \mid \mathcal L ) 
\leq \frac 12 \left( \TR \left( \Sigma^{-2} \mathcal I ( \mu \mid \mathcal L ) \right) - n + \log \det \Sigma^2 \right) . 
\]
The right-hand side of the inequality is minimal when $\Sigma = \sqrt{ \mathcal I ( \mu \mid \mathcal L )}$. This choice 
of $\Sigma$ yields the desired inequality~\eqref{eq:fisherEig1}. 
Note that the scaling proof of \eqref{eq:lsiDim2} amounts to considering 
only diagonal matrices of the form $\Sigma=\sigma \ID$ for some scalar $\sigma>0$. 
\bigskip

The first part of Theorem~\ref{thm:cov} follows from Theorem \ref{thm:fisherEig} via the Cram\'er-Rao bound:
\begin{equation}\label{eq:CR}
\C(\mu)^{-1} \preceq \IM( \mu \mid \mathcal L ) . 
\end{equation}
Indeed, recall that $\{\lambda_i\}$ and $\{\alpha_i\}$ denote the eigenvalues of $\C ( \mu)$ and 
$\IM ( \mu \mid \mathcal L )$, respectively. 
Since the map $x\mapsto \mathbf 1_{\{x>1\}} (x-1-\log x)$ is increasing on $[0,+\infty)$,
inequality~\eqref{eq:CR} imply that
\[
\begin{split}
\frac 12 \sum_{i=1}^n \mathbf{1}_{ \{\lambda_i < 1 \} } ( \lambda_i^{-1} - 1 + \log  \lambda_i )
& \le \frac 12 \sum_{i=1}^n\mathbf{1}_{\{\alpha_i>1\}}(\alpha_i-1-\log \alpha_i) \\
& \le \frac 12 \sum_{i=1}^n (\alpha_i-1-\log \alpha_i) .
\end{split}
\]
By Theorem~\ref{thm:fisherEig} this is upper bounded by the deficit $\delta ( \mu)$ 
and we obtain the first statement of Theorem~\ref{thm:cov}. 
The second part of the theorem follows from a straightforward computation which shows that  
\[
\D(\gamma_\lambda) = \frac 12 \sum_{i=1}^n \left (  \frac 1 {\lambda_i} - 1 + \log \lambda_i \right ) , 
\]
see section~\ref{sec:counter} below. The third part is also proved in section~\ref{sec:counter}. 

\section{The F\"ollmer Process}\label{sec:follmer}
Given an absolutely continuous probability measure $\mu$ on $\R^n$ we consider a stochastic process 
$(X_t)$ which is as close as possible to being a Brownian motion while having law $\mu$ 
at time $1$. Namely $X_1$ has law $\mu$, and the conditional law of 
$X$ given the endpoint $X_1$ is a Brownian bridge. 
Equivalently, the law of $X$ has density $\omega \mapsto f(\omega_1)$
with respect to the Wiener measure, where $f$ is the density of $\mu$ with respect to $\gamma$ and $\omega$ is an element of the classical Wiener space. 
In particular the process $X$ minimizes the relative entropy with respect to the Wiener measure 
among all processes having law $\mu$ at time $1$. 
This process was first considered by Schr\"odinger who was interested in the problem of minimizing the entropy with endpoint constraints, see~\cite{schrodinger} and the survey~\cite{leonard} where a nice historical account on the Schr\"odinger problem 
is given as well as the connection with optimal transportation. 

It was first observed by F\"ollmer~\cite{follmer} that the process $(X_t)$ solves the following 
stochastic differential equation: 
\[
d X_t = d B_t + \nabla \log P_{1-t} f ( X_t ) \, dt 
\] 
where $(B_t)$ is a standard Brownian motion, and $(P_t)$ is the heat semigroup, defined by 
\[
P_t h (x) = \E [ h(x+B_t) ] 
\]
for every test function $h$. We call the process $(X_t)$ the \emph{F\"ollmer process} and the process $(v_t)$ given by 
$v_t := \nabla \log P_{1-t} f ( X_t)$, the \emph{F\"ollmer drift}.

Below we recall some basic properties of this process, and we repeat the proof from \cite{lehec} of the log-Sobolev inequality based on the F\"ollmer process. We then prove more refined properties of the bridge which are needed for our stability results. Roughly, the properties (i),(ii) below correspond to the first derivative of entropy along the process while the further properties (iii),(iv),(v) correspond to the second derivative. Finally we show how the F\"ollmer process leads to natural decompositions of $\mu$.

From now on we assume that the measure $\mu$ has finite Fisher information
\[
I ( \mu \mid \gamma ) = \int_{\R^n} \vert \nabla \log f \vert^2 \, d\mu < + \infty .  
\]
\begin{proposition}\label{prop:firstDerivative}
The F\"ollmer drift $(v_t)$ has the following properties:  
\begin{itemize}
\item[(i)] The relative entropy of $\mu$ with respect to $\gamma$ satisfies 
\begin{equation}\label{eq:entropy}
H ( \mu \mid \gamma ) = \frac12 \E \left[ \int_0^1 \vert v_t \vert^2 \, dt \right].
\end{equation}
\item[(ii)] The F\"ollmer drift $(v_t)$ is a square integrable martingale.
\end{itemize}
\end{proposition}
The proof of this proposition can be found in \cite{lehec}. As was noticed in \cite{lehec}, the log-Sobolev inequality follows immediately from these properties once it is realized that
\begin{equation}\label{eq:fisher} 
\E [ \vert v_1 \vert^2 ]  = I ( \mu \mid \gamma ).
\end{equation}
Indeed, as $(v_t)$ is a martingale, $(\vert v_t\vert^2)$ is a sub-martingale so
\[
H ( \mu \mid \gamma ) = \frac12 \E \left[ \int_0^1 \vert v_t \vert^2 dt \right] 
\leq \frac 12 \E [ \vert v_1 \vert^2 ] = \frac 12 I ( \mu \mid \gamma ) . 
\]
In particular we obtained the following expression for the deficit. 
\begin{proposition}\label{prop:deficit}
Let $\mu$ be a probability measure on $\R^n$ 
with finite Fisher information and let $(v_t)$ be 
the associated F\"ollmer drift. Then  
\[
\D ( \mu ) = \frac12 \E \left[ \int_0^1 \vert v_1 - v_t \vert^2 dt \right] .  
\]
\end{proposition}
\begin{proof}
Since $(v_t)$ is a square integrable martingale we have $\E[ \langle v_1 , v_t \rangle ] = \E [ \vert v_t\vert^2 ]$ so
\[
\E \left[ \vert v_1 \vert^2 - \vert v_t \vert^2 \right] = \E \left[ \vert v_1 - v_t \vert^2 \right] . 
\] 
Combining this with \eqref{eq:entropy} and \eqref{eq:fisher} yields 
the result. 
\end{proof}

The above proof of the log-Sobolev inequality utilizes information about the first derivative of the entropy, that is, the fact that the derivative $\vert v_t \vert^2$ is a sub-martingale. In order to obtain stability estimates for the log-Sobolev inequality we need to look at the \emph{second} derivative of the entropy. This is the role of the next proposition. In what follows $(\mathcal F_t)$ is the natural filtration of the process $(X_t)$ and 
\[
{\rm cov} ( X_1 \mid \mathcal F_t ) := \E [X_1^{\otimes 2} \mid \mathcal F_t ]
 - \E [X_1 \mid \mathcal F_t ]^{\otimes 2} 
\]
denotes the conditional covariance of $X_1$ given $\mathcal F_t$. 
\begin{proposition}\label{prop:secondDerivative}
Set $Q_t = \nabla^2 P_{1-t} f(X_t)$, then
\begin{itemize}
\item[(iii)] $v_t = \int_0^t Q_s d B_s$ for all $t$.  
\item[(iv)] At least for $t<1$ the following alternative expressions for $Q_t$ hold true
\begin{align}
\label{eq:Qt1}
Q_t 
& = \frac{ \mathrm{cov} ( X_1 \mid \mathcal F_t ) }{(1-t)^2} - \frac{\ID}{1-t}  \\
\label{eq:Qt2}
&  = \E[ \nabla^2 \log  f(X_1) \mid \mathcal F_t ] + {\rm cov} ( v_1 \mid \mathcal F_t ) .
\end{align}
\item [(v)] The process $(Q_t + \int_0^t Q_s^2 \, ds)$ is a martingale. 
\end{itemize}
\end{proposition}

\begin{proof} The computation of $dv_t$ is a straightforward application of It\^o's formula.\\ 
For (iv) recall that $P_{1-t} f$ is the convolution of $f$ with some Gaussian. Putting derivatives 
on the Gaussian we get after some computations
\[
\nabla^2 \log P_{1-t} f
= - \frac{\ID}{1-t} +  \frac 1{(1-t)^2} \frac { P_{1-t} (f(x) x^{\otimes 2} )  }{P_{1-t} f} 
- \frac1{(1-t)^2} \left( \frac{  P_{1-t} (f(x) x) }{ P_{1-t} f  } \right)^{\otimes 2} .
\]
On the other hand, for every test function $u$, the following change of measure formula holds true
\[
\E [ u (X_1) \mid \mathcal F_t ] =  \frac { P_{1-t} ( uf ) (X_t)  }{P_{1-t} f (X_t)} .
\] 
This follows from the explicit expression that we have for the law of $(X_t)$. Plugging this into 
the previous display yields~\eqref{eq:Qt1}. The proof of~\eqref{eq:Qt2} is similar, only we put the 
derivatives on $f$ rather than the Gaussian when computing $\nabla^2 \log P_{1-t} f$. \\ 
To get (v) observe that by~\eqref{eq:Qt2}
\[
Q_t = {\rm martingale} - v_t \otimes v_t  . 
\]
Now since $v_t=Q_t dB_t$ we have $d (v_t \otimes v_t) = d ( {\rm martingale}) + Q_t^2 dt$. 
Hence the result. 
\end{proof}

Note that since $X_1 = B_1 + \int_0^1 v_t \, dt$ and since the expectation of $v_t$ 
is constant over time, the expectation of $v_t$ coincides with that of $\mu$. 
In addition,  it follows from~\eqref{eq:Qt1} that 
\[
\E[ Q_t ] = \E_\mu [ \nabla^2 \log f + (\nabla \log f)^{\otimes 2} ] - \E[ v_t\otimes v_t ] 
\]
for every $t$. Integrating by parts yields the following:
\begin{proposition}\label{prop:exp}
For every time $t$ we have $\E [ v_t ] = \E_\mu [ x ]$ and 
\[
\begin{split}
\E[ Q_t ] 
& = \E_\mu [ x\otimes x] - \ID - \E[ v_t\otimes v_t ] \\
& = {\rm cov} ( \mu ) - \ID - {\rm cov}  ( v_t) .  
\end{split}
\]
\end{proposition}

Other than facilitating an immediate proof of the log-Sobolev inequality, the F\"ollmer process provides a canonical decomposition of the measure $\mu$ which we now describe. Recall that
\[
 \E[h(X_1) \mid  \mathcal F_t] = \frac{ P_{1-t} (hf)(X_t) }{ P_{1-t} f(X_t) } ,
\] 
for every test function $h$. This allows to compute the density of the conditional law of $X_1$ 
given $\mathcal F_t$. Namely, let $\mu_t$ be the conditional law of $\frac { X_1-X_t } { \sqrt{1-t} }$ 
given $\mathcal F_t$.
Then 
\begin{equation}\label{eq:mut}
\mu_t ( dx ) = \frac{ f( \sqrt{1-t} \, x + X_t ) }{ P_{1-t} f(X_t ) } \, \gamma ( dx ) .
\end{equation}
\begin{lemma}
We have 
\begin{equation} \label{eq:dat}
X_1  = \int_0^1 \C( \mu_t) d B_t
\end{equation}
almost-surely, and
\begin{equation}\label{eq:defRepDef}
\D(\mu) \geq \int_0^1 \E \left[ \D ( \mu_t) \right] dt.
\end{equation}
\end{lemma}
\begin{proof}
Again $\E [X_1 \mid \mathcal F_t ] = P_{1-t} ( xf ) ( X_t )  / P_{1-t} f ( X_t)$. So 
\[
d\, \E [X_1 \mid \mathcal F_t ] =  \nabla \left( \frac{ P_{1-t} ( xf ) }{ P_{1-t} f } \right)  ( X_t ) \, d B_t. 
\]
Arguing as in the proof of (iv) we get
\[
\nabla \left( \frac{ P_{1-t} ( xf ) }{ P_{1-t} f } \right)  ( X_t ) = \frac { {\rm cov} ( X_1 \mid \mathcal F_t ) }{1-t} = {\rm cov } ( \mu_t ) , 
\]
which proves~\eqref{eq:dat}. 
For the inequality~\eqref{eq:defRepDef}, observe that by~\eqref{eq:mut}
\[
\D( \mu_t ) = \frac{1-t}2 \E [\vert \nabla \log f(X_1) \vert^2 \mid \mathcal F_t ]
 - \E [\log f ( X_1 ) \mid  \mathcal F_t ] + \log P_{1-t} f( X_t) . 
\]
Also, by It\^o's formula 
\[
d \log P_{1-t} f( X_t) = v_t \, dB_t + \frac 12 \vert v_t \vert^2 \, dt . 
\]
Putting everything together we get  
\[
\E \left[\D( \mu_t ) \right]  = \frac 12 \int_t^1 \E \left[\vert v_1 - v_s  \vert^2 \right] \, ds .
\]
Thus, by Proposition~\ref{prop:deficit}
\[
\int_0^1 \E \left [\D (\mu_t) \right ] dt
= \frac 12 \int_0^1 s \,  \E [ \vert v_1 - v_s \vert^2 ]\, ds 
\leq \D ( \mu ) . \qedhere 
\]
\end{proof}

\begin{remark}
At this stage it maybe worth noticing that the measure-valued process $(\mu_t)$ 
coincides with a simplified version of the stochastic localization process 
of the first named author~\cite{eldan0}.
\end{remark}

\section{Comparison theorems}\label{sec:comp}
In this section we prove Theorems~\ref{thm:fisherEig} and~\ref{thm:cov} via the F\"ollmer process.

\begin{proof}[Proof of Theorem~\ref{thm:fisherEig}]
Because the result is invariant by scaling we can assume without loss of 
generality that $\C (\mu)$ is strictly smaller than the identity. Let $m(t) = - \E [ Q_t ]$. 
We know from Proposition~\ref{prop:exp} that 
\begin{equation}\label{eq:idmt}
m(t) = - {\rm cov} (\mu) +\ID + {\rm cov} (v_t) . 
\end{equation}
This shows in particular that $m(t)$ is positive definite. 
Item (v) of Proposition~\ref{prop:secondDerivative} shows that $\frac{d}{dt}m(t) \succeq m(t)^2$. 
Since $m(t)$ is positive definite this amounts to $\frac d{dt} m(t)^{-1} \preceq - \ID$. 
We use this information to compare $m(t)$ with $m(1)$. 
We get 
\begin{equation}\label{eq:compstep2}
m(t) \preceq \left( m(1)^{-1} + (1-t) \ID \right)^{-1} . 
\end{equation}
Let $\tilde f$ be the density of $\mu$ with respect to the Lebesgue measure and observe 
that   
\[
\begin{split}
m(1) & = - \E [ Q_1 ] = - \E_\mu [ \nabla^2 \log f ] \\
& =  - \E_\mu [ \nabla^2 \log \tilde f ] - \ID \\
& = \mathcal I ( \mu \mid \mathcal L ) - \ID . 
\end{split}
\]
Taking the trace in~\eqref{eq:compstep2} and using Proposition~\ref{prop:exp} thus gives 
\[
- \E_\mu [ \vert x\vert^2 ] + n + \E[ \vert v_t \vert^2 ] 
\leq \sum_{i=1}^n \frac 1 { (\alpha_i-1)^{-1} + 1-t }  ,
\] 
where the $\alpha_i$ are the eigenvalues of $\mathcal I ( \mu \mid \mathcal L )$.
Integrating between $0$ and $1$ and applying item (i) of Proposition~\ref{prop:firstDerivative}
yields
\[
- \E_\mu [ \vert x\vert^2 ]  +  n + 2\, H ( \mu \mid \gamma )  
\leq \sum_{i=1}^n  \log ( \alpha_i ) .
\]
Lastly, a straightforward computation shows that the left hand side equals 
$2 H ( \mu \mid \mathcal L ) - 2 H ( \gamma \mid \mathcal L)$.  
\end{proof}
\begin{proof}[Proof of Theorem~\ref{thm:cov}]
Consider the orthogonal decomposition $\C(\mu) = \sum_{i=1}^n \lambda_i u_i^{\otimes 2}$ where $u_i$ are unit orthogonal vectors,
and again let $m(t) = -\E [Q_t ]$. Recall~\eqref{eq:idmt}, which shows in particular that $m(0) = - \C ( \mu ) + \ID$. Fix $i \in [n]$ such that $\lambda_i < 1$ and denote $\theta = u_i$. Note that 
\[
\langle \theta, m(0) \theta \rangle = 1-\lambda_i > 0 . 
\]
Moreover, we have
\begin{align*}
\frac{d}{dt} \bigl \langle \theta, m(t) \theta \bigr \rangle \geq \bigl \langle \theta, m(t)^2 \theta \bigr \rangle \geq \bigl \langle \theta, m(t) \theta \bigr \rangle^2.
\end{align*}
Since the function $g(t) = \frac{1}{1/c - t}$ solves the ordinary differential equation $\frac{d}{dt}g(t) = g(t)^2$ with the boundary condition $g(0) = c$, an application of Gr\^onwall's inequality gives 
\[
\langle \theta, m(t) \theta \rangle \geq \frac{1}{\langle \theta, m(0) \theta \rangle^{-1} - t} =\frac{1}{(1-\lambda_i)^{-1} - t} . 
\]
Summing up over all $i$ such that $\lambda_i < 1$, we have 
\[
\frac d{dt} \E[\vert v_t\vert^2 ] = \TR \left ( m(t)^2 \right ) \geq \sum_{i=1}^n \frac{\mathbf{1}_{ \{ \lambda_i < 1 \}  } }{ { \left((1-\lambda_i)^{-1} -t\right)^2 }}.
\]
Integrating this between $t$ and $1$, we obtain
\[
\E [ \vert v_1 \vert^2 ] - \E [ \vert v_t\vert^2 ] 
\geq \sum_{i=1}^n \mathbf{1}_{ \{ \lambda_i < 1 \}  } \left ( \frac 1 {\lambda_i} - 1 - \frac 1{ (1-\lambda_i)^{-1} - t } \right ). 
\]
Now we integrate between $0$ and $1$ and we use Proposition~\ref{prop:deficit}.
We get
\[
\D ( \mu ) \geq \frac 12 \sum_{i=1}^n \mathbf{1}_{ \{ \lambda_i < 1 \}  } \left (  \frac 1 {\lambda_i} - 1 + \log \lambda_i \right ),
\]
which is the desired inequality.
\end{proof}

\section{Decompositions into mixtures}\label{sec:mix}
In this section we prove Theorems \ref{thm:dim} and \ref{thm:uncor}. 

\begin{proof}[Proof of Theorem~\ref{thm:dim}]
The idea of the proof is to show that for any $t$, the transportation distance between $X_1$ and the sum of the independent random vectors $\E[X_1\mid\mathcal F_t]+(B_1-B_t)$ can be controlled by the deficit. Optimizing over $t$ yields the theorem. 
\\
The map 
\[
t\mapsto \E \left[ \vert v_1\vert^2  - \vert v_t \vert^2  \right]
\]
is a non-increasing function since $(\vert v_t \vert^2)$ is a sub-martingale. Hence by Proposition \ref{prop:deficit} and as $(v_t)$ is a martingale,
\begin{equation}\label{eq:stepdivis}
\D ( \mu ) = \frac 12 \E \left[ \int_0^1 \vert v_1 - v_t \vert^2 dt \right]=\frac 12 \int_0^1 \E \left[ \vert v_1\vert - \vert v_t \vert^2\right] dt \geq \frac t2 \E \left[ \vert v_1 - v_t \vert^2 \right] 
\end{equation}
for every $t\in [0,1]$. Let $Y_t = \E [ X_1 \mid  \mathcal F_t ] + B_1 - B_t$ and note that since $B_1 - B_t$ is independent 
of $\mathcal F_t$, the random vector $Y_t$ has law $\nu_t * \gamma_{0,1-t}$
where $\nu_t$ is the law of $\E[X_1 \mid \mathcal F_t]$. Hence since $X_1$ has law $\mu$ and 
\[
\E [ X_1 \mid \mathcal F_t ] = X_t + (1-t) v_t,
\]
we get by Jensen's inequality that
\begin{align*}
\W_2^2 ( \mu , \nu_t * \gamma_{0,1-t} ) & \leq \E \left[\vert X_1 - Y_t \vert^2 \right] 
= \E \left[\left\vert \int_t^1 (v_s - v_t) ds \right\vert^2 \right] \leq (1-t) \int_t^1 \E \left[\vert v_s - v_t \vert^2 \right] ds \\
& \leq (1-t)^2 \E [ \vert v_1 - v_t \vert^2 ]\leq  \E [ \vert v_1 - v_t \vert^2 ].
\end{align*}  
Combining this with \eqref{eq:stepdivis} yields
\[
\W_2^2 ( \mu , \nu_t * \gamma_{0,1-t} ) \leq \frac 2t \D ( \mu ) . 
\]
This inequality gives the distance between $\mu$ and a mixture of Gaussians but with the wrong covariance. To remedy that we must pay a dimensional price. By the triangle inequality for $\W_2$ and the 
fact that $\W_2^2 ( \gamma_{0,1-t} , \gamma_{0,1} ) \leq (1 - \sqrt{1-t})^2 n \leq t^2 n$, 
we get
\begin{align*}
\W_2 ( \mu , \nu_t * \gamma ) 
& \leq \W_2 ( \mu , \nu_t * \gamma_{0,1-t} ) + \W_2 (  \nu_t * \gamma_{0,1-t} ,  \nu_t * \gamma ) \\
& \leq \W_2 ( \mu , \nu_t * \gamma_{0,1-t} ) +  \W_2 ( \gamma_{0,1-t} , \gamma ) \\
& \leq \sqrt{ \frac{2\D ( \mu )}{t}}+ \sqrt{n} t. 
\end{align*}
If $\D(\mu) \leq n$, choosing $t=\left(\frac{\D(\mu)}{n}\right)^{\frac{1}{3}}$
in the previous display gives 
\[
\W_2 ( \mu , \nu_t * \gamma ) \leq (\sqrt 2 +1) n^\frac{1}{6}\D( \mu)^{\frac{1}{3}} ,
\]
which in turn yields the desired inequality \eqref{eq:dimension}. If on the contrary $\D(\mu)\geq n$, inequality \eqref{eq:dimension}
holds with $\nu = \mu$, simply because $\W_2 ( \mu , \mu * \gamma ) = \sqrt n$.
If $\D(\mu)=0$ the argument shows that $\mu = \nu_0 * \gamma$, where $\nu_0$ is the Dirac point mass at $\E [X_1]$. 
\end{proof}

\begin{proof}[Proof of Theorem~\ref{thm:uncor}]
The starting point of the proof is identity~\eqref{eq:dat}:
\[
X_1  = \int_0^1 \C(\mu_t) d B_t.
\]
The idea is then to extract from this identity two processes $(Y_t),(Z_t)$ close to each other in transportation distance such that $Z_1\sim \gamma$. We then write $X_1=Y_1+W$ for some random vector $W$ and show that $\E[\langle Y,W\rangle]=0$. The requirement $Z_1\sim \gamma$ is enforced by ensuring that the quadratic variation of $(Z_t)$ satisfies $[Z]_1=\ID$. 

We start with some notation. Let $M$ be an $n \times n$ symmetric matrix and let $M = \sum_{i =1}^n\kappa_i\, u_i \otimes u_i$ be its eigenvalue decomposition. We then set $M_+ := \sum_{i =1}^n  \max ( \kappa_i , 0 ) \, u_i \otimes u_i$ and similarly $\max ( M , \ID ) = \sum_{i =1}^n  \max ( \kappa_i , 1 ) \, u_i \otimes u_i$. Using Theorem \ref{thm:cov} together with the fact that $\frac{1}{x} - 1 + \log(x) \geq \frac{1}{2} (x-1)^2$ for all $x \in (0,1]$, we conclude that for every measure $\nu$, one has
\[
\D(\nu) \geq \frac{1}{2} \TR \left[ \left( \ID - \C(\nu) \right)_+^2  \right].
\]
Using this bound and inequality \eqref{eq:defRepDef} we get,
\begin{equation}\label{eq:deficitpos}
\D(\mu) \geq \frac{1}{2}  \E\left[\int_0^1  
\TR \left[ \left( \ID - \C(\mu_t) \right)_+^2 \right]dt\right].
\end{equation}
Next we will write the right-hand side above as roughly the difference 
in transportation distance between the random vectors $Y_1$ and $Z_1$ mentioned above.

For convenience, define $A_t :=  \C(\mu_t)$. We now define a random process $(C_t)$ 
taking values in the set of symmetric matrices as follows.
Set $C_0 = 0$ and  
\[
d C_t =  \max ( A_t^2 , \ID ) \, dt , \quad t \in [0, \tau_1)
\]
where $\tau_1$ is the first time the largest eigenvalue of $C_t$ 
hits the value $1$. Notice that $t \ID \preceq C_t$ on $[0,\tau_1)$, 
so $\tau_1 \leq 1$. If $C_{\tau_1} \neq \ID$, which implies that $\tau_1 < 1$, 
we let $O_1$ be the orthogonal projection 
onto the range of $C_{\tau_1} - \ID$, and set 
\[
d C_t = O_1 \max ( A_t^2 , \ID ) O_1 \, dt , \quad t \in [\tau_1 , \tau_2 ) 
\]
where $\tau_2$ is the first time the largest eigenvalue of $O_1 C_t O_1$ hits the value $1$. 
If $C_{\tau_2} \neq \ID$ we let $O_2$ be the projection onto the range of $C_{\tau_2}-\ID$ and proceed 
similarly, and so on, until the first time $\tau_k$ such that  $C_{\tau_k} = \ID$. 
On $[\tau_k,1]$ we let $dC_t = 0$ and thus $C_t = \ID$. 
To sum up, the matrix $C_t$ satisfies $0 \preceq C_t \preceq \ID$ for all $t\in [0,1]$, $C_1 = \ID$ and 
\[
d C_t = L_t \max ( A_t^2 , \ID ) L_t \, dt ,  
\]
where $L_t$ is the orthogonal projection onto the range of $C_t-\ID$. 

Next, consider the processes $(Y_t),(Z_t)$ defined by
\[
Y_0 = Z_0 = 0, ~~ d Y_t = L_t A_t d B_t, ~~ d Z_t = L_t  \max \left ( A_t, \ID \right ) d B_t.
\]
and note that
\[
[Z]_t = \int_0^t L_s \max \left ( A_s^2, \ID \right ) L_s ds = C_t.
\]
This implies that $[Z]_1 = \ID$ almost surely so $Z_1 \sim \gamma$.
On the other hand, we have by \eqref{eq:deficitpos} and It\^o's isometry,
\begin{align*}
\E [|Y_1 - Z_1|^2] & = \E \left[\int_0^1 \TR \left[  L_t \left ( \max \left ( A_t, \ID \right ) - A_t \right )^2  L_t \right] dt\right] \\
& \leq \E \left[\int_0^1 \TR \left [ \left (  \max \left ( A_t, \ID \right ) - A_t  \right )^2 \right ] dt\right] \\
& = \E \left[\int_0^1 \TR \left [ \left ( \left ( \ID - A_t \right )_+ \right )^2 \right ] dt\right]\leq 2 \D(\mu).
\end{align*}
Letting $\nu$ be the law of $Y_1$, we thus get $\W_2^2(\nu, \gamma) \leq 2 \D(\mu)$. Now define the random vector $W := \int_0^1 \left ( A_t - L_t A_t \right ) d B_t$ so by \eqref {eq:dat}, $X_1=Y_1+W$. 
It remains to show that $\E[ \langle Y,  W \rangle ]=0$. This is again a consequence of It\^o's isometry: 
\begin{align*}
\E[ \langle Y,  W \rangle ] 
& = \E \left[\int_0^1 \TR  \left ( L_t A_t \left ( A_t - L_t A_t \right )^T \right )  dt\right] \\
& = \E \left[\int_0^1 \TR  \left ( L_t A_t^2 - L_t A_t^2 L_t \right) dt\right] = 0 
\end{align*}
since $L_t = L_t^2$. This completes the proof.
\end{proof}

\section{Counterexamples to stability}\label{sec:counter}
In this section we provide simple counterexamples 
to the stability of the logarithmic Sobolev inequality
with respect to the Wasserstein distance, thus proving Theorem \ref{thm:unstable} as well as the third part of Theorem \ref{thm:cov}. The standard Gaussian on $\R$ is denoted by $\gamma$ and for $a\in \R, s\ge 0$ we let $\gamma_{a,s}$ be the Gaussian centered at $a$ with variance $s$. Our counterexamples are nothing more than Gaussian mixtures. For such measures, the following two lemmas provide a lower bound on the Wasserstein $p$-distance to translated Gaussians, and an upper bound on the log-Sobolev deficit. The combination of these two lemmas will prove Theorems \ref{thm:unstable} and \ref{thm:cov}. We start with the upper bound on the log-Sobolev deficit. 

\begin{lemma}\label{lem:mixDef}
Let $a,b \in \R$, and $\sigma, t\in [0,1]$. Then 
\[
\D \left( (1-t) \gamma_{a,\sigma} + t \gamma_{b,\sigma} \right) 
\leq \frac 14 \left( \sigma^{-1} -1\right)^2 -(1-t) \log ( 1-t) - t\log t.
\]
\end{lemma} 
\begin{proof}
Let $\varphi (t) = t \log t + (1-t)\log(1-t)$ and $\mu,\nu$ be probability measures on $\R$. The lemma follows immediately by combining the estimates 
\begin{equation}\label{eq:mixture}
\D \left( (1-t) \mu + t \nu \right)  \leq (1-t) \D \left( \mu \right) +t \D \left( \nu \right) -\varphi (t)
\end{equation}
and 
\begin{equation}\label{eq:gaussDef}
\D(\gamma_{0, \sigma} ) \leq\frac 14 \left( \sigma^{-1} - 1 \right)^2 . 
\end{equation}
(when $\sigma \leq 1$) and using the fact that $\D$ is invariant under translations.
\\
The validity of (\ref{eq:mixture}) follows immediately from the combination of the convexity estimates 
\[
I \left( (1-t) \mu + t \nu \mid \gamma \right) \leq (1-t) I ( \mu \mid \gamma ) + t I ( \nu \mid \gamma)
\]
and 
\[
H \left( (1-t) \mu + t\nu \mid \gamma \right) \geq (1-t) H ( \mu \mid \gamma ) +
t H ( \nu \mid \gamma ) + \varphi (t).
\]
The convexity of the Fisher information is a well-known fact, 
it is a direct consequence of the convexity of the map  
$(x,y) \mapsto y^2 /x$ on $(0,\infty)\times \R$. For the second inequality, let $f$ and $g$ 
be the respective densities of $\mu$ and $\nu$ with respect to $\gamma$ and use the fact that the logarithm is increasing
to write 
\[
\left( (1-t)f+tg \right) 
\log\left( (1-t)f+tg \right) 
\geq 
(1-t)f \log \left( (1-t)f \right) + t g \log (t g ) . 
\]
Integrating with respect to $\gamma$ yields the result.\\
For the estimate~\eqref{eq:gaussDef}, a direct computation shows that 
$H(\gamma_{0,\sigma} \mid \gamma) = \frac{1}{2}\left( \sigma - 1 - \log \sigma \right)$ 
and $I(\gamma_{0,\sigma} \mid \gamma)= (\sigma-1)^2/\sigma$, so that
\[
\D( \gamma_{0,\sigma} ) =\frac{1}{2}\left( \sigma^{-1}  - 1 + \log ( \sigma ) \right) . 
\]
We conclude using $x-1-\log x \leq (x-1)^2 / 2$ for $x\geq 1$. 
\end{proof}

\begin{proof}[Proof of Proposition~\ref{prop:shannon}]
When $\sigma =1$, Lemma \ref{lem:mixDef} can be rewritten $\D ( p * \gamma ) \leq S (p)$ 
for any probability measure $p$ in $\R$ which is a combination of two Dirac point masses. 
The argument can easily be generalized to any discrete probability measure $p$, and to any dimension,
proving Proposition~\ref{prop:shannon}. 
\end{proof}

\begin{proof}[Proof of the third part of Theorem~\ref{thm:fisherEig}]
Note that 
\[
\mathrm{var}( (1-t) \gamma_{a,1} + t \gamma_{b,1} ) = 1 + t(1-t) (b-a)^2 . 
\]
Set $\mu_k = \left( 1-\frac 1k \right) \gamma_{0,1} + \frac 1k \gamma_{k^2,1}$. 
Then ${\rm var}(\mu_k) \to \infty$. On the other hand $\D(\mu_k)\to 0$ by Lemma \ref{lem:mixDef}. 
\end{proof}

Next we move on to the lower bound on the Wasserstein distance.
\begin{lemma}\label{lem:w1}
Let $a,b\in \R$, $\sigma \in (0,1]$, $t\in [0,1]$ and 
let $\mu = (1-t) \gamma_{a, \sigma} + t\gamma_{b, \sigma}$. Suppose that
\begin{equation}\label{eq:tailassump}
\min (t,1-t) \geq 2 \exp \left( - \frac { (b-a)^2 }{32} \right) .
\end{equation}
Then, for every $p\geq 1$
\[
\inf_{m\in \R} \left\{  \W_p^p ( \mu , \gamma_{m,1} )
\right\} \geq  \min(t, 1-t)\,  \frac{ \vert b-a\vert^p }{ 4^{p+1} }   .
\]
\end{lemma}

\begin{proof}
Let $m\in \R$ and suppose without loss of generality that $|a-m| \leq |b-m|$ and that $b>a$. Define $z = m + \sqrt{2 \log\left(\frac{2}{t}\right)}$ and note that the assumption \eqref{eq:tailassump} together with the fact that $b-m \geq \frac{1}{2} |b-a|$ implies that $b-z \geq \frac{1}{4} |b-a|$. Now, by a standard Gaussian tail estimate we have $\gamma_{m,1} \bigl ( [z, \infty) \bigr) \leq \frac{t}{4}$. On the other hand 
\[
\mu ( [b, \infty) ) \geq  t \gamma_{b, \sigma^2} ( [b, \infty) ) = \frac{t}{2} .
\]
Therefore, in order to transport $\gamma_{m,1}$ to $\mu$, at least $t/4$ unit of mass to the left of 
$z$ should move to the right of $b$. As a result
\[
\W_p^p(\mu , \gamma_{m,1}) \geq \frac{t}{4} (b-z)^p \geq \frac{t}{4} \left (\frac{|b-a|}{4} \right )^p ,
\]
which yields the result. 
\end{proof}

\begin{proof}[Proof of Theorem \ref{thm:unstable}]
For the first part of the theorem we shall work in dimension $1$ but the result extends to any dimension by taking the tensor product of the one dimensional example by a standard Gaussian. Consider the sequence of measures $(\mu_k)$ given by 
\[
\mu_k = \left( 1-\frac 1k \right) \gamma_{0,1} + \frac 1k \gamma_{k^2,1}.   
\]
Lemmas \ref{lem:mixDef} and \ref{lem:w1} imply that $\D(\mu_k) \to 0$ and 
$\inf_{m\in \R} \left\{ \W_1(\mu_k, \gamma_{m,1} ) \right\} \to \infty$.
\\
For the second part of the theorem, define $\mu_k = (1-t) \gamma_{a,\sigma} + t\gamma_{b,\sigma}$ 
with
\[
t= k^{-3/2} , \quad a= - k^{-1} , \quad b = - \frac{1-t}t a , \quad \sigma = 1 - t(1-t) (b-a)^2 .  
\]
It is straightforward to check that $\mu_k$ is isotropic. 
Since $(b-a)^2 = a^2 /t^2 = k$, the hypothesis~\eqref{eq:tailassump} is satisfied for large enough $k$
and Lemma~\ref{lem:w1} gives
\[
\inf_m \left\{ \W_2^2 (\mu_k, \gamma_{m,1}) \right\} \geq t \frac{(b-a)^2}{64} = \frac 1{ 64 \sqrt k} . 
\]
On the other hand, we have $\sigma = 1 - k^{-1/2} + o(k^{-1/2})$ so that Lemma~\ref{lem:mixDef} gives
\[
\D(\mu_k) \leq \frac 1 k + o \left( \frac 1 k \right) .  
\]
Set $n(k) = \left \lfloor k^{3/4} \right \rfloor$. Since both 
the deficit and $\W_2^2$ behave additively when taking tensor products
we have 
\[
\inf_{m\in \R^{n(k)}} \left\{ \W_2^2 \left (\mu_k^{\otimes n(k)} , \gamma_{m,\mathrm{Id}_{n(k)}} \right) \right\} 
= \Omega (  k^{1/4} ) = \Omega ( n(k)^{1/3} ) 
\]
and $\D \left ( \mu_k^{\otimes n(k)}  \right )  = O ( k^{-1/4} ) = O ( n(k)^{-1/3} )$.
\end{proof}

%
%
%%%%%%%%%%%
%
%

\end{document}